\documentclass[11pt, reqno]{amsart}
\usepackage{amsmath, amsthm, amssymb, amsfonts}
\usepackage[colorlinks=true, citecolor=blue, linkcolor=blue]{hyperref}
\usepackage[margin=1in]{geometry}

\newtheorem{theorem}{Theorem}[section]
\newtheorem{corollary}[theorem]{Corollary}
\newtheorem{lemma}[theorem]{Lemma}

\theoremstyle{definition}

\newtheorem{remark}[theorem]{Remark}

\title{A Simple Counting Argument for Dense Linear Hypergraphs}

\author{Lior Gishboliner}
\address{Department of Mathematics, University of Toronto, Toronto, ON, Canada}
\email{lior.gishboliner@utoronto.ca} 

\author{J\'ozsef Solymosi}
\address{Department of Mathematics, University of British Columbia, Vancouver, BC, Canada \and \'Obuda University, Budapest, Hungary}
\email{solymosi@math.ubc.ca}
\thanks{Research supported in part by an NSERC Discovery Grant.}


\dedicatory{\textnormal{\textit{Dedicated to the memory of Vera T. S\'os}}}

\begin{document}

\begin{abstract}
In connection to the Brown-Erd\H{o}s-S\'os conjecture, we give a short local averaging proof of a density theorem for linear uniform hypergraphs. Let $r \ge 3$, $k \ge 3$, and suppose that $n \ge (r-2)(k-2)+1$. If $H$ is a linear $r$-uniform hypergraph on $n$ vertices and 
\[|E(H)| \geq \frac{k-2}{r^2((r-2)(k-2)+1)}n^2 + \frac{n}{r},\]
then $H$ contains $k$ edges spanning at most $(r-2)k+3$ vertices. In the standard linear-density normalization, this gives the asymptotic density threshold $c \geq \frac{r-1}{r} \cdot \frac{k-2}{(r-2)(k-2)+1} + o(1)$. 
In particular, this yields a simple proof of the large-uniformity form of the Brown-Erd\H{o}s-S\'os theorem, due to Keevash and Long.
In the case of triple systems, our bound becomes 
$c \geq \frac{2(k-2)}{3(k-1)} + o(1)$, improving upon a bound of $\frac{4}{5}$ due to Santos and Tyomkyn.

\end{abstract}

\maketitle

\section{Introduction}

An $r$-uniform hypergraph is \textit{linear} if any two distinct edges intersect in at most one vertex. For integers $s > k \ge 2$, an \textit{$(s, k)$-configuration} is a collection of $k$ edges spanning at most $s$ vertices. The celebrated Brown-Erd\H{o}s-S\'os problem \cite{BROWN1973_1, BROWN1973_2} asks for the minimum intermediate edge density required to force the existence of such localized, dense configurations in uniform hypergraphs. 

For a linear $r$-uniform hypergraph $H$ on $n$ vertices, the natural quadratic scale for the edge count is normalized by the maximum possible number of linear edges. We write
\begin{equation}\label{eq:density_def}
|E(H)| = c \frac{\binom{n}{2}}{\binom{r}{2}}
\end{equation}
and refer to $c \in [0,1]$ as the \textit{linear density} of $H$.

Recently, Santos and Tyomkyn \cite{SANTOS2025} proved that for every fixed $k \ge 4$, any sufficiently large linear triple system with linear density $c > 4/5$ contains a $(k+3, k)$-configuration. Their calculations also gave an asymptotic density threshold of $4/7$ for $k=4$. According to the Brown-Erd\H{o}s-S\'os conjecture $c\rightarrow 0$ as $n$ goes to infinity. The conjecture is widely open, the only known case is $k=3$, by the famous result of Ruzsa and Szemer\'edi \cite{RuzsaSzemeredi1978}. So, it is a rigtful question asking what is the least $c$ where the $(k+3, k)$-configuration is guaranteed. For some progress on approximate versions of this problem, see \cite{CGLS,JANZER2025,Shapira_Tyomkyn 2}.

In this paper, we present a counting argument that yields a sharper threshold of 
\[\frac{2(k-2)}{3(k-1)} + o(1)\]
for triple systems. More importantly, this bound is derived not as an asymptotic limit, but as a direct consequence of an explicit, finite theorem valid across all admissible parameters $r, k,$ and $n$.

Our main result is the following finite structural theorem.

\begin{theorem}\label{thm:main}
Let $r \ge 3$ and $k \ge 3$, and suppose $n \ge (r-2)(k-2)+1$. Let $H$ be a linear $r$-uniform hypergraph on $n$ vertices. If
\begin{equation}\label{eq:main_bound}
|E(H)| \geq \frac{k-2}{r^2((r-2)(k-2)+1)}n^2 + \frac{n}{r},
\end{equation}
then $H$ contains an $((r-2)k+3, k)$-configuration.
\end{theorem}

Expressed in terms of the linear density defined in \eqref{eq:density_def}, Theorem \ref{thm:main} yields the following immediate corollary.

\begin{corollary}\label{cor:density_general}
Let $r \ge 3$, $k \ge 3$, and $n \ge (r-2)(k-2)+1$. Let $H$ be a linear $r$-uniform hypergraph on $n$ vertices with linear density $c$. If
\[c \geq \frac{r-1}{r} \cdot \frac{k-2}{(r-2)(k-2)+1} \cdot \frac{n}{n-1} + \frac{r-1}{n-1},\]
then $H$ contains an $((r-2)k+3, k)$-configuration. In particular, the asymptotic density threshold \nolinebreak is 
\[\frac{r-1}{r} \cdot \frac{k-2}{(r-2)(k-2)+1} + O_r\left(\frac{1}{n}\right),\]
with the error term uniform in $k$ across the admissible range.
\end{corollary}

Specializing Corollary \ref{cor:density_general} to the case of triple systems ($r=3$), the vertex span simplifies to $(3-2)k+3 = k+3$, yielding an explicit threshold for small configurations.

\begin{corollary}\label{cor:triples}
Let $k \ge 3$ and $n \ge k-1$, and let $H$ be a linear triple system on $n$ vertices. If
\[|E(H)| > \frac{k-2}{9(k-1)}n^2 + \frac{n}{3},\]
then $H$ contains a $(k+3, k)$-configuration. Equivalently, in the linear-density normalization, it is sufficient to assume
\[c > \frac{2(k-2)}{3(k-1)} \cdot \frac{n}{n-1} + \frac{2}{n-1}.\]
Thus, the asymptotic threshold is $\frac{2(k-2)}{3(k-1)} + O(\frac{1}{n})$ uniformly for all $k$. For $k=4$, this yields a threshold of $\frac{4}{9} + O(\frac{1}{n})$, improving upon the $\frac{4}{7}$ bound established in \cite{SANTOS2025}.
\end{corollary}

Finally, because the explicit density threshold in Corollary \ref{cor:density_general} is bounded above by a function of $r$ that decays independently of $k$, our finite theorem provides a short, elementary proof of the large-uniformity Brown-Erd\H{o}s-S\'os theorem proved via probabilistic absorption by Keevash and Long \cite{KEEVASH2020}.

\begin{corollary}\label{cor:keevash_long}
For every $c_0 > 0$, there exists an integer $r_0 = r_0(c_0)$ such that the following holds. For every $r \ge r_0$ and every $k \ge 3$, there exists an $n_0 = n_0(r, k, c_0)$ such that every linear $r$-uniform hypergraph on $n \ge n_0$ vertices with linear density at least $c_0$ contains an $((r-2)k+3, k)$-configuration. 
\end{corollary}
\noindent
In particular, we get a short proof of the Ramsey-type results of \cite{Shapira_Tyomkyn 1}, which follow from Corollary \nolinebreak \ref{cor:keevash_long}.

\section{The Component Lemma}

The sole non-averaging ingredient in our argument is an elementary structural bound on the connected components of a hypergraph. 

\begin{lemma}[Component Lemma]\label{lem:component}
Let $s \ge 2$ and $l \ge 2$. Let $G$ be an $s$-uniform hypergraph such that $e(G) \ge l$ and
\begin{equation}\label{eq:lemma_hyp}
e(G) > \frac{l-1}{(s-1)(l-1)+1}v(G).
\end{equation}
Then there exists a subset of vertices $U \subseteq V(G)$ such that $e_G(U) \ge l$ and $|U| \le (s-1)l+1$.
\end{lemma}

\begin{proof}
We begin with a standard preliminary observation regarding connected hypergraphs: if $C$ is a connected $s$-uniform hypergraph containing at least $q$ edges, we may select $q$ of its edges $f_1, \dots, f_q$ in a \textit{connected order}. That is, $f_1$ is chosen arbitrarily, and for every $2 \le i \le q$, the edge $f_i$ satisfies 
\[f_i \cap \left( \bigcup_{a=1}^{i-1} f_a \right) \neq \emptyset.\]
Because each successive edge adds at most $s-1$ new vertices to the union, the total vertex span of these $q$ edges is bounded above by
\begin{equation}\label{eq:connected_span}
\left| \bigcup_{i=1}^q f_i \right| \le s + (q-1)(s-1) = (s-1)q + 1.
\end{equation}

Now, let $C_1, \dots, C_m$ be the connected components of $G$, ordered by non-increasing edge density:
\begin{equation}\label{eq:ordering}
\frac{e(C_1)}{|C_1|} \ge \frac{e(C_2)}{|C_2|} \ge \dots \ge \frac{e(C_m)}{|C_m|}.
\end{equation}
(Isolated vertices, if present, are treated as individual components with zero edges). 

If any single component $C_i$ contains at least $l$ edges, we may apply the observation in \eqref{eq:connected_span} to a connected $l$-edge subgraph within $C_i$. This immediately yields a subset of vertices $U$ with $e_G(U) \ge l$ and $|U| \le (s-1)l+1$, satisfying the lemma. We may therefore assume for the remainder of the proof that 
\begin{equation}\label{eq:small_comps}
e(C_i) \le l-1 \quad \text{for all } 1 \le i \le m.
\end{equation}

Because the total edge count satisfies $e(G) \ge l$, there exists a minimal integer $j \ge 1$ such that the prefix sum of the component edges reaches or exceeds $l$:
\[\sum_{i=1}^j e(C_i) \ge l.\]
Define the prefix edge count $q$ and prefix vertex count $p$ strictly prior to index $j$ as
\[q := \sum_{i=1}^{j-1} e(C_i), \qquad p := \sum_{i=1}^{j-1} |C_i|.\]
By the minimality of $j$, we have $q \le l-1$. We now construct our candidate edge set by taking all $q$ edges from the components $C_1, \dots, C_{j-1}$, alongside a collection of precisely $l-q$ edges selected in a connected order from component $C_j$. Together, these form a set of exactly $l$ edges in $G$. Invoking \eqref{eq:connected_span} on the portion inside $C_j$, the total number of vertices spanned by this configuration is at most
\[p + (s-1)(l-q) + 1.\]
If this quantity is bounded above by $(s-1)l+1$, the proof is complete. Subtracting $(s-1)(l-q)+1$ from both sides, this success condition is algebraically equivalent to the inequality $p \le (s-1)q$. 

We assume for contradiction that the configuration fails, which forces the strict reverse inequality:
\begin{equation}\label{eq:p_large}
p > (s-1)q.
\end{equation}

The real-valued function $\phi(x) = \frac{x}{(s-1)x+1}$ is strictly monotonically increasing across the domain $x \ge 0$. Because $q \le l-1$, applying $\phi$ to $q$ yields
\[\frac{q}{(s-1)q+1} \le \frac{l-1}{(s-1)(l-1)+1}.\]
Furthermore, rewriting \eqref{eq:p_large} as $p \ge (s-1)q+1$ allows us to bound the prefix density:
\begin{equation}\label{eq:prefix_bound}
\frac{q}{p} \le \frac{q}{(s-1)q+1} \le \frac{l-1}{(s-1)(l-1)+1}.
\end{equation}

Note that $|C_j| > (s-1)e(C_j)$. Suppose to the contrary that $|C_j| \le (s-1)e(C_j)$, meaning its local density satisfies $e(C_j)/|C_j| \ge \frac{1}{s-1}$. By the non-increasing density ordering \eqref{eq:ordering}, this would imply that every preceding component shares this lower bound:
\[\frac{e(C_i)}{|C_i|} \ge \frac{1}{s-1} \implies |C_i| \le (s-1)e(C_i) \quad \text{for all } i < j.\]
Summing this inequality over all $i < j$ yields $p \le (s-1)q$, directly contradicting our assumption in \eqref{eq:p_large}. 

Thus, we must have $|C_j| > (s-1)e(C_j)$, which forces $e(C_j)/|C_j| < \frac{1}{s-1}$. Using the component ordering \eqref{eq:ordering} once more, this strict upper bound propagates to all subsequent components:
\[\frac{e(C_i)}{|C_i|} < \frac{1}{s-1} \implies |C_i| > (s-1)e(C_i) \quad \text{for all } i \ge j.\]
Because the sizes and edge counts of hypergraph components are strictly integral, the strict inequality $|C_i| > (s-1)e(C_i)$ tightens to the discrete bound 
\[|C_i| \ge (s-1)e(C_i) + 1 \quad \text{for all } i \ge j.\]
Applying the monotonicity of $\phi(x)$ to the component counts $e(C_i) \le l-1$, we achieve a uniform upper bound on the density of every tail component:
\begin{equation}\label{eq:tail_bound}
\frac{e(C_i)}{|C_i|} \le \frac{e(C_i)}{(s-1)e(C_i)+1} \le \frac{l-1}{(s-1)(l-1)+1} \quad \text{for all } i \ge j.
\end{equation}

We reach our contradiction by reassembling the global edge density of $G$. Partitioning the vertices and edges at index $j$ and applying the fractional bounds \eqref{eq:prefix_bound} and \eqref{eq:tail_bound}, we obtain
\begin{align*}
e(G) &= q + \sum_{i=j}^m e(C_i) \\
&\le \left[ \frac{l-1}{(s-1)(l-1)+1} \right] p + \sum_{i=j}^m \left[ \frac{l-1}{(s-1)(l-1)+1} \right] |C_i| \\
&= \frac{l-1}{(s-1)(l-1)+1} \left( p + \sum_{i=j}^m |C_i| \right) = \frac{l-1}{(s-1)(l-1)+1}v(G).
\end{align*}
This violates the strict density lower bound assumed in \eqref{eq:lemma_hyp}, proving that the initial failure assumption \eqref{eq:p_large} is false. The lemma follows.
\end{proof}

\section{Proof of the Main Results}
\noindent
With Lemma \ref{lem:component} in hand, we proceed to the proof of the main structural theorem.

\begin{proof}[Proof of Theorem \ref{thm:main}]
Let $m = |E(H)|$, and let $d(x)$ denote the degree of vertex $x \in V(H)$. We first claim that there exists a ``heavily intersected" base edge $e_0 = \{x_1, \dots, x_r\} \in E(H)$ whose incident vertex degrees satisfy
\begin{equation}\label{eq:heavy_edge}
\sum_{i=1}^r d(x_i) \ge \frac{r^2 m}{n}.
\end{equation}
To verify this, we evaluate the sum of the vertex degrees over all edges in $H$. Re-indexing the summation over the vertices and applying the standard Cauchy-Schwarz inequality, we have
\[\sum_{e \in E(H)} \sum_{x \in e} d(x) = \sum_{x \in V(H)} d(x)^2 \ge \frac{1}{n} \left( \sum_{x \in V(H)} d(x) \right)^2 = \frac{1}{n}(rm)^2 = \frac{r^2 m^2}{n}.\]
Because this total is distributed across precisely $m$ edges, the Pigeonhole Principle guarantees the existence of at least one edge $e_0$ achieving the average asserted in \eqref{eq:heavy_edge}. Fix this edge $e_0$.

We now define an auxiliary $(r-1)$-uniform hypergraph $G$ with vertex set $V(G) = V(H) \setminus e_0$. The edges of $G$ are constructed via projection: for every edge $f \in E(H) \setminus \{e_0\}$ that intersects $e_0$, the linearity of $H$ guarantees that the intersection consists of exactly one vertex, say $f \cap e_0 = \{x_i\}$. We project $f$ down to the auxiliary edge
\[f' := f \setminus \{x_i\} \in \binom{V(G)}{r-1}.\]

Crucially, because $H$ is linear and $r \ge 3$, no two distinct edges of $H$ collapse to the same auxiliary edge in $G$. Indeed, if two distinct edges $f_1, f_2 \neq e_0$ yielded the same set $S$ outside $e_0$, then $|S| = r-1 \ge 2$. If they intersected $e_0$ at the same vertex $x_i$, they would be identical ($f_1 = S \cup \{x_i\} = f_2$). If they intersected $e_0$ at distinct vertices $x_a \neq x_b$, then $f_1$ and $f_2$ would share the entire set $S$ of size at least $2$, violating the linearity of $H$. 

Hence, the mapping $f \mapsto f'$ is a bijection from the set of edges outside $e_0$ intersecting $e_0$ to the edge set of $G$. Summing the contributions across the vertices of $e_0$, we obtain the exact edge count:
\begin{equation}\label{eq:aux_edges}
e(G) = \sum_{i=1}^r (d(x_i) - 1) = \left( \sum_{i=1}^r d(x_i) \right) - r \ge \frac{r^2 m}{n} - r.
\end{equation}

Define the rational parameter $\alpha := \frac{k-2}{(r-2)(k-2)+1}$. Substituting our bound for $m$ from \eqref{eq:main_bound} into the lower bound \eqref{eq:aux_edges}, we have:
\begin{align*}
e(G) &\ge \frac{r^2}{n} \left[ \frac{k-2}{r^2((r-2)(k-2)+1)}n^2 + \frac{n}{r} \right] - r \\
&= \left[ \frac{k-2}{(r-2)(k-2)+1} \right]n + r - r = \alpha n.
\end{align*}
Thus, $e(G) \geq \alpha n$. Because the auxiliary vertices exclude the base edge, $v(G) \le n - r < n$, immediately yielding the scale-free density lower bound 
\begin{equation}\label{eq:aux_density_met}
e(G) > \alpha v(G).
\end{equation}
The assumption $n \ge (r-2)(k-2)+1$ ensures that $\alpha n \ge k-2$. Bounding $e(G) > k-2$ over the integers strictly forces $e(G) \ge k-1$.

We can now apply Lemma \ref{lem:component} to the auxiliary hypergraph $G$, setting its uniformity to $s = r-1$ and its target edge threshold to $l = k-1$. Notice that the coefficient required by the lemma matches our parameter $\alpha$ exactly:
\[\frac{l-1}{(s-1)(l-1)+1} = \frac{(k-1)-1}{((r-1)-1)((k-1)-1)+1} = \frac{k-2}{(r-2)(k-2)+1} = \alpha.\]
Because $e(G) \ge k-1$ and condition \eqref{eq:aux_density_met} holds, Lemma \ref{lem:component} provides a subset of auxiliary vertices $U \subseteq V(G)$ containing at least $k-1$ edges of $G$, bounded in size by
\[|U| \le (s-1)l + 1 = (r-2)(k-1) + 1.\]

We lift this structure back to the original hypergraph by defining $W := U \cup e_0 \subseteq V(H)$. The $k-1$ auxiliary edges in $G[U]$ correspond to $k-1$ distinct edges of $H$ that intersect $e_0$ and are fully contained within $W$. Adjoining the base edge $e_0$ itself, we have identified a collection of at least $k$ distinct edges of $H$ residing entirely within $W$. Accounting for the base vertices, the total vertex span of this configuration is bounded by
\[|W| \le |U| + |e_0| \le (r-2)(k-1) + 1 + r = (r-2)k + 3.\]
This confirms that $H$ contains an $((r-2)k+3, k)$-configuration, completing the proof.
\end{proof}

\begin{proof}[Proof of Corollaries \ref{cor:density_general} and \ref{cor:triples}]
To establish Corollary \ref{cor:density_general}, we express the absolute edge count as $m = c \frac{n(n-1)}{r(r-1)}$. Expanding the inequality from Theorem \ref{thm:main}, we require
\[c \frac{n(n-1)}{r(r-1)} \geq \frac{k-2}{r^2((r-2)(k-2)+1)}n^2 + \frac{n}{r}.\]
Multiplying both sides by the reciprocal scalar $\frac{r(r-1)}{n(n-1)}$ isolates $c$ directly:
\[c \geq \frac{r-1}{r} \cdot \frac{k-2}{(r-2)(k-2)+1} \cdot \frac{n}{n-1} + \frac{r-1}{n-1}.\]
To verify the asymptotic claim, observe that for all $r \ge 3$ and $k \ge 3$, the core coefficient satisfies $\frac{r-1}{r} \frac{k-2}{(r-2)(k-2)+1} \le 1$. Writing $\frac{n}{n-1} = 1 + \frac{1}{n-1}$, the secondary expansion gets absorbed into the $O_r(1/n)$ error term. Setting $r=3$ reduces the statement to Corollary \ref{cor:triples}.
\end{proof}

\begin{proof}[Proof of Corollary \ref{cor:keevash_long}]
Consider the function $\psi(x) = \frac{x}{(r-2)x+1}$. Because $\psi(x)$ is strictly increasing for $x \ge 0$, evaluating it at the finite integer $x = k-2$ strictly bounds it below its horizontal asymptote:
\[\frac{k-2}{(r-2)(k-2)+1} < \lim_{x \to \infty} \psi(x) = \frac{1}{r-2} \quad \text{for all } k \ge 3.\]
Consequently, the main density threshold from Corollary \ref{cor:density_general} is bounded above by
\[\frac{r-1}{r} \cdot \frac{1}{r-2} = \frac{r-1}{r(r-2)}.\]
Given any target density $c_0 > 0$, choose an integer $r_0$ sufficiently large so that $\frac{r_0-1}{r_0(r_0-2)} < c_0$. Because this expression decreases monotonically toward zero as $r \to \infty$, the strict bound $\frac{r-1}{r(r-2)} < c_0$ holds for all $r \ge r_0$ independently of $k$. Finally, for any fixed $r \ge r_0$ and $k \ge 3$, choosing $n_0$ sufficiently large overcomes the $O_r(\frac{1}{n})$ finite terms, ensuring that any linear $r$-uniform hypergraph on $n \ge n_0$ vertices with density at least $c_0$ triggers the configuration.
\end{proof}

\section{Acknowledgements}
\noindent
Each of the authors is supported by an NSERC Discovery Grant.

\end{document}